\documentclass{amsart}
\usepackage[utf8]{inputenc}

\usepackage{amssymb}
\usepackage{amsmath}
\usepackage{amsthm}
\usepackage[pdftex]{graphicx}
\usepackage{tikz}
\usepackage{url}
\usepackage{color}
\usepackage{xifthen}
\usepackage{xcolor}
\usepackage{mathtools}
\usepackage[backend=biber,style=alphabetic,sorting=nty,doi=false,isbn=false,url=false,eprint=true]{biblatex}
\usepackage{hyperref}
\usepackage{geometry}
\renewbibmacro{in:}{}
\renewcommand{\P}{\mathbb{P}}
\newcommand{\Q}{\mathbb{Q}}

\newcommand{\C}{\mathbb{C}}

\newcommand{\dom}{\operatorname{dom}}

\newcommand\forces{\Vdash}

\newcommand{\frakc}{\mathfrak{c}}
\newcommand{\frakb}{\mathfrak{b}}
\newcommand{\frakd}{\mathfrak{d}}
\newcommand{\frakv}{\mathfrak{v}}
\newcommand{\Pow}{\mathcal{P}}
\newcommand{\non}{\operatorname{non}}
\newcommand{\cov}{\operatorname{cov}}

\newcommand{\cof}{\operatorname{cof}}
\newcommand{\Cof}{\mathbf{Cof}}
\newcommand{\Cov}{\mathbf{Cov}}
\newcommand{\D}{\mathbf{D}}
\newcommand{\Lc}{\mathbf{Lc}}
\newcommand{\wLc}{\mathbf{wLc}}
\newcommand{\nul}{\mathsf{null}}
\newcommand{\meager}{\mathsf{meager}}

\newcommand{\from}{\colon}
\newcommand{\cf}{\operatorname{cf}}
\newcommand{\KT}{\operatorname{KT}}
\newcommand{\SAT}{\mathrm{SAT}}
\newcommand{\Con}{\operatorname{Con}}
\newcommand{\LangL}{\mathcal{L}}
\newcommand{\scrA}{\mathcal{A}}
\newcommand{\scrB}{\mathcal{B}}
\newcommand{\mcf}{\mathfrak{mcf}}
\newcommand{\CH}{\mathrm{CH}}

\newcommand{\ZFC}{\mathrm{ZFC}}

\newcommand{\seq}[1]{{\langle#1\rangle}}
\newcommand{\quot}[1]{\ulcorner #1 \urcorner}
\DeclarePairedDelimiter\abs{\lvert}{\rvert}
\DeclarePairedDelimiterX{\norm}[1]{\lVert}{\rVert}{#1}
\DeclarePairedDelimiter\floor{\lfloor}{\rfloor}

\renewcommand\emptyset{\varnothing}
\renewcommand\subset{\subseteq}
\renewcommand{\setminus}{\smallsetminus}

\newcommand{\needtocheck}[1][]{%
	\ifthenelse{\equal{#1}{}}{%
		\textcolor{blue}{[NeedToCheck]}%
	}{%
		\textcolor{blue}{[NeedToCheck: #1]}%
	}%
}

\newcommand{\todo}[1][]{%
	\ifthenelse{\equal{#1}{}}{%
		\textcolor{red}{[TODO]}%
	}{%
		\textcolor{red}{[TODO: #1]}%
	}%
}

\theoremstyle{definition}
\newtheorem{thm}{Theorem}[section]
\newtheorem*{thm*}{Theorem}
\newtheorem{defi}[thm]{Definition}
\newtheorem*{defi*}{Definition}
\newtheorem{lem}[thm]{Lemma}
\newtheorem*{lem*}{Lemma}
\newtheorem{fact}[thm]{Fact}
\newtheorem*{fact*}{Fact}
\newtheorem*{formula*}{Formula}

\newtheorem*{prop*}{Proposition}

\newtheorem*{exm*}{Example}
\newtheorem{rmk}[thm]{Remark}
\newtheorem*{rmk*}{Remark}
\newtheorem{cor}[thm]{Corollary}
\newtheorem*{cor*}{Corollary}
\newtheorem*{notation*}{Notation}
\newtheorem*{convention*}{Convention}
\newtheorem{question}[thm]{Question}

\title{Keisler's Theorem and Cardinal Invariants}
\author{Tatsuya Goto}
\date{\today}
\address{
	\newline
	Graduate School of Information Science\newline
	Nagoya University\newline
	Furo-cho, Chikusa-ku, Nagoya 464-8601\newline
	JAPAN
}

\email{goto.tatsuya@k.mbox.nagoya-u.ac.jp}

\newcommand{\myfigure}[3]{ %
	\begin{tikzpicture}[xscale=#1, yscale=#2] %
		\node (ch) at (0, -1) {$\mathrm{CH}$}; %
		\node (satc) at (0, -3) {$\SAT(\frakc)$}; %
		\node (ktc) at (0, -5) {$\KT(\frakc)$}; %
		\node (sat1) at (4, -3) {$\SAT(\aleph_1)$}; %
		\node (kt1) at (4, -5) {$\KT(\aleph_1)$}; %
		\node (sat0) at (8, -3) {$\SAT(\aleph_0)$}; %
		\node (kt0) at (8, -5) {$\KT(\aleph_0)$}; %
		\node (ma) at (8, -1) {$\mathrm{MA}$}; %
		\node (b) at (4, -6) {$\frakb = \aleph_1$}; %
		\node (cex) at (8, -6) {$\cov(\nul) \le \frakd$}; %
		\node[align=center] (covmreg) at (8, -2) {$\cov(\meager) = \frakc \land\ 2^{<\frakc} = \frakc$}; %
		\node (covm) at (8, -4) {$\cov(\meager) = \frakc$}; %
		\node[align=center] (covmcf) at (2, -4) {$\cov(\meager) = \frakc$ \\ $\land\ \cf(\frakc) = \aleph_1$}; %
		\draw[double distance=2pt] (ch) -- (satc); %
		\draw[double distance=2pt] (satc) -- (ktc); %
		\draw[->] (ch) -- (ma); %
		\draw[double distance=2pt, line width=#3] (satc) -- (sat1); %
		\draw[->] (sat1) -- (sat0); %
		\draw[->] (ktc) -- (kt1); %
		\draw[->] (kt1) -- (kt0); %
		\draw[->] (sat1) -- (kt1); %
		\draw[->] (ma) -- (covmreg); %
		\draw[double distance=2pt, line width=#3] (covmreg) -- (sat0); %
		\draw[->, line width=#3] (sat0) -- (covm); %
		\draw[->] (covm) -- (kt0); %
		\draw[->, line width=#3] (kt1) -- (b); %
		\draw[->, line width=#3] (kt0) -- (cex); %
		\draw[->] (ch) -- (covmcf); %
		\draw[->] (covmcf) -- (kt1); %
		\draw[->] (covmcf) -- (covm); %
	\end{tikzpicture} %
}

\begin{document}
	\maketitle
	
	\begin{abstract}
		We consider several variants of Keisler's isomorphism theorem.
		We separate these variants by showing implications between them and cardinal invariants hypotheses.
		We characterize saturation hypotheses that are stronger than Keisler's theorem with respect to models of size $\aleph_1$ and $\aleph_0$ by $\CH$ and $\cov(\meager) = \frakc \land 2^{<\frakc} = \frakc$ respectively.
		We prove that Keisler's theorem for models of size $\aleph_1$ and $\aleph_0$ implies $\frakb = \aleph_1$ and $\cov(\nul) \le \frakd$ respectively.
		As a consequence, Keisler's theorem for models of size $\aleph_0$ fails in the random model.
		We also show that for Keisler's theorem for models of size $\aleph_1$ to hold it is not necessary that $\cov(\meager)$ equals $\frakc$.
	\end{abstract}

	\section{Introduction}

	The method of ultrapower is one of the most important ways to construct models.  
	Ultrapowers are models obtained by properly equating the elements of product sets of the models using ultrafilters.
	We consider the problem when there exists an ultrafilter $U$ on $\omega$ such that for two models $\scrA, \scrB$ in a countable language $\LangL$, the respective ultrapowers $\scrA^\omega/U, \scrB^\omega/U$ are isomorphic.	
	Since ultrapowers are elementary extensions of original models, if $\scrA^\omega/U$ and $\scrB^\omega/U$ are isomorphic, then $\scrA$ and $\scrB$ must be elementarily equivalent.
	Keisler showed, under CH, conversely if $\scrA$ and $\scrB$ are elementarily equivalent and have size $\le \frakc$, then for every ultrafilter $U$ over $\omega$, $\scrA^\omega/U$ and $\scrB^\omega/U$ are isomorphic.
	The purpose of this paper is to give necessary conditions and sufficient conditions for when Keisler's theorem holds in a model where CH does not hold, and to separate the variants of Keisler's theorem using those conditions.

	\begin{convention*}
		All ultrafilters considered in this paper are nonprincipal.
	\end{convention*}

	\begin{defi}
		Let $\kappa$ be a cardinal.
		\begin{enumerate}
			\item We say $\KT(\kappa)$ holds if for every countable language  $\LangL$ and $\LangL$-structures $\mathcal{A}, \mathcal{B}$ of size $\le \kappa$ with $\mathcal{A} \equiv \mathcal{B}$, there exists an ultrafilter $U$ over $\omega$ such that $\mathcal{A}^\omega/U \simeq \mathcal{B}^\omega/U$.
			\item We say $\SAT(\kappa)$ holds if there exists an ultrafilter $U$ over $\omega$ such that for every countable language $\LangL$ and every sequence of $\LangL$-structures $(\mathcal{A}_i)_{i\in\omega}$ with each $\mathcal{A}_i$ of size $\le \kappa$, $\prod_{i \in \omega} \mathcal{A}_i / U$ is saturated.
		\end{enumerate}
	\end{defi}
	
	$\SAT(\kappa)$ implies $\KT(\kappa)$ for every $\kappa$.
	Keisler \cite{keisler1961ultraproducts} proved $\mathrm{CH} \Rightarrow \SAT(\frakc)$.
	Recently, Golshani and Shelah \cite{golshani2021keislershelah} proved the converse $\KT(\frakc) \Rightarrow \mathrm{CH}$. In fact, they showed $\neg \KT(\aleph_2)$.
	They also proved that $\cov(\meager) = \frakc \land \cf(\frakc) = \aleph_1$ implies $\KT(\aleph_1)$.
	Another classical result is the theorem by Ellentuck--Rucker \cite{ellentuck1972martin} which shows that $\operatorname{MA}(\sigma\text{-centered})$ implies $\SAT(\aleph_0)$.
	Moreover, Shelah \cite{shelah1992vive} showed $\Con(\neg \KT(\aleph_0))$ by showing that $\frakd < \mathfrak{v}^\forall$ implies $\neg \KT(\aleph_0)$ and that the former is consistent.
	
	In this paper, we prove the implications indicated by thick lines in Figure \ref{implication-figure}.
	
	\begin{figure}[h]
		\centering
		\[ \myfigure{0.9}{1.2}{0.6mm} \]
		\caption{Implications}
		\label{implication-figure}
	\end{figure}
	
	In the rest of this section, we recall basic notions related to cardinal invariants.
	
	\begin{defi}
		\begin{enumerate}
			\item If $X, Y$ are sets and $R$ is a subset of $X \times Y$, we call a triple $(X, Y, R)$ a \textit{relational system}.
			\item For a relational system $\mathcal{A} = (X, Y, R)$, define $\mathcal{A}^\perp = (Y, X, \hat{R})$, where $\hat{R} = \{ (y, x) \in Y \times X : \neg (x \mathrel{R} y)\})$.
			\item For a relational system $\mathcal{A} = (X, Y, R)$, define $\norm{\mathcal{A}} = \min \{ \abs{B} : B \subset Y \land (\forall x \in X)(\exists y \in B) (x \mathrel{R} y) \}$.
			\item For relational systems $\mathcal{A} = (X, Y, R), \mathcal{B} = (X', Y', S)$, we call a pair $(\varphi, \psi)$ a Galois--Tukey morphism from $\mathcal{A}$ to $\mathcal{B}$ if $\varphi \from X \to X'$, $\psi \from Y' \to Y$ and $(\forall x \in X)(\forall y \in Y')(\varphi(x) \mathrel{S} y \implies x \mathrel{R} \psi(y))$.
		\end{enumerate}
	\end{defi}
	
	\begin{fact}[{{\cite[Theorem 4.9]{blass2010combinatorial}}}]
		If there is a Galois--Tukey morphism $(\varphi, \psi)$ from $\mathcal{A}$ to $\mathcal{B}$, then $\norm{\mathcal{A}} \le \norm{\mathcal{B}}$ and $\norm{\mathcal{B}^\perp} \le \norm{\mathcal{A}^\perp}$. 
	\end{fact}
	
	\begin{defi}
		\begin{enumerate}
			\item For $f, g \in \omega^\omega$, define $f <^* g$ iff $(\forall^\infty n) (f(n) < g(n))$.
			\item For $c \in (\omega+1)^\omega, h \in \omega^\omega$, define $\prod c = \prod_{n \in \omega} c(n)$ and $S(c, h) = \prod_{n \in \omega} [c(n)]^{\le h(n)}$.
			\item For $x \in \prod c$ and $\varphi \in S(c, h)$, define $x \in^* \varphi$ iff $(\forall^\infty n)(x(n) \in \varphi(n))$ and define $x \in^\infty \varphi$ iff $(\exists^\infty n)(x(n) \in \varphi(n))$.
		\end{enumerate}
	\end{defi}
	
	\begin{defi}
		\begin{enumerate}
			\item Define $\D = (\omega^\omega, \omega^\omega, <^*)$, $\frakd = \norm{\D}$ and $\frakb = \norm{\mathbf{D}^\perp}$.
			\item  For a poset $(P, <)$, define a relational system $\Cof(P, <)$ by $\Cof(P, <) = (P, P, <)$. Then we have $\cf(P, <) = \norm{\Cof(P, <)}$.
			\item For $c \in (\omega+1)^\omega, h \in \omega^\omega$, define $\Lc(c, h) = (\prod c, S(c, h), \in^*)$, $\frakc^\forall_{c, h} = \norm{\Lc(c, h)}$ and $\frakv^\forall_{c, h} = \norm{\Lc(c, h)^\perp}$.
			\item Define $\wLc(c, h) = (\prod c, S(c, h), \in^\infty)$, $\frakc^\exists_{c, h} = \norm{\wLc(c, h)}$ and $\frakv^\exists_{c, h} = \norm{\wLc(c, h)^\perp}$.
			\item For an ideal $I$ on $X$, define $\Cov(I) = (X, I, \in)$, $\cov(I) = \norm{\Cov(I)}$ and $\non(I) = \norm{\Cov(I)^\perp}$.
		\end{enumerate}
	\end{defi}
	
	\begin{defi}
		\begin{enumerate}
			\item Define $\frakv^\forall = \min \{ \frakv^\forall_{c, h} : c, h \in \omega^\omega, \lim h = \infty \}$.
			\item Define $\frakc^\exists = \min \{ \frakc^\exists_{c, h} : c, h \in \omega^\omega, \sum h / c < \infty \}$.
		\end{enumerate}
	\end{defi}

	\begin{fact}[{{\cite[Proposition 1.1.5]{garcia2016model}}}]\label{ultrapowershavecontinuumsize}
			Let $\seq{\scrA_i : i \in \omega}$ be a sequence of structures in a language $\LangL$ such that each $\scrA_i$ has size $\le \frakc$.
			Let $U$ be an ultrafilter over $\omega$.
			Then the ultraproduct $\prod_{i \in \omega} \scrA_i / U$ has size either finite or $\frakc$.
	\end{fact}
	
	\section{$\SAT(\aleph_1)$ and $\KT(\aleph_1)$}
	
	In this section, we prove that $\SAT(\aleph_1)$ is equivalent to $\CH$ and that $\KT(\aleph_1)$ implies $\frakb = \aleph_1$.
	
	\begin{thm}\label{sataleph1}
		$\SAT(\aleph_1)$ implies CH.
	\end{thm}
	\begin{proof}
		Assume $\SAT(\aleph_1)$ and $\neg \CH$.
		Take an ultrafilter $U$ over $\omega$ that witnesses $\SAT(\aleph_1)$.
		Let $\scrA_\ast = (\omega_1, <)^\omega / U$.
		For $\alpha < \omega_1$, put $\alpha_\ast = [\seq{\alpha, \alpha, \alpha, \dots}]$.
		Define a set $p$ of formulas with a free variable $x$ by
		\[
		p = \{ \quot{\alpha_\ast < x} : \alpha < \omega_1 \}.
		\]
		This $p$ is finitely satisfiable and the number of parameters occuring in $p$ is $\aleph_1 < \frakc = \abs{\scrA_\ast}$ by $\neg \CH$.
		Thus, by $\SAT(\aleph_1)$, we can take $f \from \omega \to \omega_1$ such that $[f]$ realizes $p$.
		Put $\beta = \sup_{n \in \omega } f(n)$.
		Now we have $\{ n \in \omega : \beta <  f(n) \} \in U$ and this contradicts the definition of $\beta$.
	\end{proof}
	
	\begin{defi}
		Let $\mcf = \min \{ \cf(\omega^\omega/U) : U \text{ an ultrafilter over } \omega \}$.
	\end{defi}
	
	\begin{lem}[{{\cite[Claim 2.2]{golshani2021keislershelah}}}]\label{cofinalityomega1}
		Let $\scrA$ be a structure in a language $\LangL = \{<\}$.
		Suppose that $a \in \scrA$ has cofinality $\omega_1$.
		Let $U$ be an ultrafilter over $\omega$.
		Then $a_\ast = [\seq{a, a, a, \dots}]$ has cofinality $\omega_1$ in $\scrA^\omega / U$.
	\end{lem}
	\begin{proof}
		Take an increasing cofinal sequence $\seq{x_\alpha : \alpha < \omega_1}$ of points in $\scrA$ below $a$.
		Then $\seq{x_\alpha^* : \alpha < \omega_1}$ is an increasing cofinal sequence in $\scrA_\ast$, where $x_\alpha^* = [\seq{x_\alpha, x_\alpha, x_\alpha, \dots}]$ for each $\alpha < \omega_1$.
		This can be shown by regularity of $\omega_1$.
	\end{proof}
	
	\begin{lem}[{{\cite[Claim 2.4]{golshani2021keislershelah}}}]\label{qhomogeneous}
		Let $U$ be an ultrafilter over $\omega$ and $\scrB_\ast = (\Q, <)^\omega/U$.
		Then for every $a, b \in \scrB_\ast$, there is an automorphism on $\scrB_\ast$ that sends $a$ to $b$.
	\end{lem}
	\begin{proof}
		Consider the map $F \colon \Q^3 \to \Q$ defined by $F(x, y, z) = x - y + z$.
		Then we have
		\[
		(\forall y, z \in \Q) (\text{the map } x \mapsto F(x, y, z) \text{ is an automorphism on } (\Q, <) \text{ that sends $y$ to $z$}).
		\]
		This statement can be written by a first-order formula in the language $\LangL' = \{<, F\}$.
		Thus the same statement is true in $(\Q, <, F)^\omega/U$. The map $F_\ast : \scrB_\ast^3 \to \scrB_\ast$ induced by $F$ satisfies that
		\[
		(\forall 	y, z \in \scrB_\ast) (\text{the map } x \mapsto F(x, y, z) \text{ is an automorphism on } (\scrB_\ast, <) \text{ that sends $y$ to $z$}). \qedhere
		\]
	\end{proof}
	
	\begin{thm}\label{ktaleph1b}
		$\KT(\aleph_1)$ implies $\mcf = \aleph_1$.
	\end{thm}
	
	\begin{proof}
		This proof is based on \cite[Theorem 2.1]{golshani2021keislershelah}.
		Assume that $\mcf \ge \aleph_2$. We shall show $\neg \KT(\aleph_1)$.
		
		Let $\LangL = \{ < \}$, $\scrA = (\Q, <)$ and $\scrB = (\Q + ((\omega_1+1) \times \Q_{\ge 0}), <_\scrB)$. Here $<_\scrB$ is defined by a lexicographical order and a disjoint union order.
		$\scrA$ and $\scrB$ are dense linear ordered sets, so by completeness of DLO, we have $\scrA \equiv \scrB$.
		Take an ultrafilter $U$ over $\omega$.
		Put $\scrA_* = \scrA^\omega / U, \scrB_* = \scrB^\omega / U$.
		
		There is a point $a$ in $\scrB$ such that $\cf(\scrB_a) = \aleph_1$, where $\scrB_a = \{ x \in \scrB : x < a \}$.
		Then $a_* \in \scrB_*$ has cofinality $\aleph_1$ by Lemma \ref{cofinalityomega1}.
		Here $a_* = [\seq{a, a, a, \dots}]$.
		On the other hand, we shall show every point in $\scrA_*$ has cofinality $\ge \mcf$.
		If we do this, since we assumed $\mcf \ge \aleph_2$, we will have $\scrA_* \not \simeq \scrB_*$.
		
		By Lemma \ref{qhomogeneous}, it suffices to consider the point $0_* = [\seq{0, 0, 0, \dots}]$.
		Since $\Q$ is symmetrical, we consider $\cf((\Q_{>0})^\omega/U, >_U)$.
		
		Now we construct a Galois--Tukey morphism $(\varphi, \psi) \colon \Cof(\omega^\omega/U) \to \Cof((\Q_{>0})^\omega/U, >_U)$ by 
		\begin{align*}
			\varphi &\colon \omega^\omega/U \to (\Q_{>0})^\omega/U; [f] \mapsto [\seq{1/(f(n)+1) : n \in \omega}], \\
			\psi &\colon (\Q_{>0})^\omega/U \to \omega^\omega/U; [g] \mapsto [\seq{\floor{1/g(n)-1} : n \in \omega}].
		\end{align*}
		So we have $\cf((\Q_{>0})^\omega/U, >_U) \ge \cf(\omega^\omega/U, <_U)$.
		
		Thus we have $\cf((\Q_{>0})^\omega/U, >_U) \ge \mcf$. We are done.
	\end{proof}
	
	\begin{cor}
		$\KT(\aleph_1)$ implies $\frakb = \aleph_1$.
	\end{cor}
	\begin{proof}
		This follows from Theorem \ref{ktaleph1b} and the fact that $\frakb \le \mcf$.
	\end{proof}
	
	\begin{rmk}
		The combination of Theorem \ref{ktaleph1b} and \cite[Theorem 3.2]{golshani2021keislershelah} gives the implication $(\cov(\meager) = \frakc \land \cf(\frakc) = \aleph_1) \implies \frakb = \aleph_1$.
		This is immediate from the facts that $\cov(\meager) \le \frakd$ and $\frakb \le \cf(\frakd)$.
	\end{rmk}
	
	\section{$\SAT(\aleph_0)$ and $\KT(\aleph_0)$}
	
	In this section, we first briefly mention consistency of $\KT(\aleph_0) + \neg \KT(\aleph_1)$. And we prove that $\SAT(\aleph_0)$ is equivalent to $\cov(\meager) = \frakc \land 2^{<\frakc} = \frakc$.
	
	\begin{thm}\label{covmktaleph0}
		$\cov(\meager) = \frakc$ implies $\KT(\aleph_0)$.
	\end{thm}
	\begin{proof}
		By the same proof as \cite[Theorem 3.2]{golshani2021keislershelah}.
	\end{proof}
	
	\begin{cor}
		Assume $\Con(\ZFC)$. Then $\Con(\ZFC + \KT(\aleph_0) + \neg \KT(\aleph_1))$.
	\end{cor}
	\begin{proof}
		$\mathrm{MA} + \neg \mathrm{CH}$ implies $\KT(\aleph_0) \land \neg \KT(\aleph_1)$ by Theorem \ref{ktaleph1b} and \ref{covmktaleph0}.
	\end{proof}
	
	\begin{fact}[{{\cite[Lemma 2.4.2]{bartoszynski1995set}}}]
		$\cov(\meager) = \mathfrak{v}^\exists_{\seq{\omega : n \in \omega}, \mathrm{id}}$.
		In other words, $\cov(\meager) \ge \kappa$ holds iff $(\forall X \subset \omega^\omega \text{ of size } <\kappa) (\exists S \in \prod_{i \in \omega} [\omega]^{\le i}) (\forall x \in X) (\exists^\infty n) (x(n) \in S(n))$ holds.
	\end{fact}
	
	\begin{thm}
		$\SAT(\aleph_0)$ implies $\cov(\meager) = \frakc$.
	\end{thm}
	\begin{proof}
		Take an ultrafilter $U$ that witnesses $\SAT$.
		Fix $X \subset \omega^\omega$ of size $< \frakc$.
		Define a language $\LangL$ by $\LangL = \{\subset\}$ and for each $i \in \omega$, define a $\LangL$-structure $\scrA_i$ by $\scrA_i = ([\omega]^{\le i}, \subset)$.
		For each $x \in \omega^\omega$, let $S_x = \seq{\{x(i)\} : i \in \omega}$.
		In the ultraproduct $\scrA_* = \prod_{i\in\omega} \scrA_i / U$, define a set $p$ of formulas of one free variable $S$ by
		\[
		p = \{ \quot{[S_x] \subset S} : x \in X \}.
		\]
		This $p$ is finitely satisfiable.
		In order to check this, let $x_0, \dots, x_n$ be finitely many members of $X$.
		Define $S$ by $S(m) = \{ x_0(m), \dots, x_n(m) \}$ for $m \ge n$. We don't need to care about $S(m)$ for $m < n$.
		Then this $S$ satisfies $[S_{x_i}] \subset [S]$ for all $i \le n$. 
		Moreover, the number of parameters of $p$ is $< \frakc$.
		
		So by $\SAT(\aleph_0)$, we can take $[S] \in \mathcal{A}_*$ that realizes $p$.
		Then $S$ fulfills $(\forall x \in X) (\{ n \in \omega : x(n) \in S(n) \} \in U)$.
		Thus $(\forall x \in X) (\exists^\infty n) (x(n) \in S(n))$.
	\end{proof}
	
	\begin{thm}
		$\SAT(\aleph_0)$ implies $2^{<\frakc} = \frakc$.
	\end{thm}
	\begin{proof}
		Take an ultrafilter $U$ over $\omega$ that witnesses $\SAT(\aleph_0)$.
		Fix $\kappa < \frakc$.
		
		Put $\LangL = \{\subset\}$ and define an $\LangL$-structure $\scrA$ by $\scrA = ([\omega]^{<\omega}, \subset)$. Put $\scrA^* = \scrA^\omega/U$.
		
		Define a map $\iota \colon \omega^\omega/U \to \scrA^*$ by $\iota([x]) = [\seq{\{x(n)\} : n \in \omega}]$.
		By Fact \ref{ultrapowershavecontinuumsize}, we have $\abs{\omega^\omega/U} = \frakc$. Take a subset $F$ of $\omega^\omega/U$ of size $\kappa$.
		
		For each $X \subset F$, let $p_X$ be a set of formulas with a free variable $z$ defined by
		\[
		p_X = \{ \quot{\iota(y) \subset z} : y \in X \} \cup \{ \quot{\iota(y) \not \subset z} : y \in F \setminus X \}
		\]
		
		Each $p_X$ is finitely satisfiable. In order to check this, take $[x_0], \dots, [x_n] \in X$ and $[y_0], \dots, [y_m] \in F \setminus X$.
		Put $z(i) = \{ x_0(i), \dots, x_n(i) \}$. Then $\iota([x_0]), \dots, \iota([x_n]) \subset_U [z]$.
		In order to prove $\iota([y_j]) \not \subset_U [z]$ for each $j \le m$, suppose that $\{ i \in \omega : y_j(i) \in z(i) \} \in U$.
		Then for each $i \in \omega$, there is a $k_i \le n$ such that $\{ i \in \omega : y_j(i) = x_{k_i}(i)\} \in U$.
		Then there is a $k \le n$ such that $\{ i \in \omega : y_j(i) = x_k(i)\} \in U$.
		This implies $[y_j] = [x_k]$, which is a contradiction.
		
		By $\SAT(\aleph_0)$, for each $X \subset F$, take $[z_X] \in \mathcal{A}^*$ that realizes $p_X$.
		For $X, Y \subset F$ with $X \ne Y$, we have $[z_X] \ne [z_Y]$.
		So $2^\kappa = \abs{\{[z_X] : X \subset F\}} \le \abs{\mathcal{A}^*} = \frakc$.
		Therefore we have proved $2^{<\frakc} = \frakc$. 
	\end{proof}
	
	\begin{thm}
		$\cov(\meager) = \frakc \land 2^{<\frakc} = \frakc$ implies $\SAT(\aleph_0)$.
	\end{thm}
	\begin{proof}
		This proof is based on \cite[Theorem 1]{ellentuck1972martin}.
		
		Let $\seq{b_\alpha : \alpha < \frakc}$ be an enumeration of $\omega^\omega$.
		Let $\seq{(\LangL_\xi, \mathcal{B}_\xi, \Delta_\xi) : \xi < \frakc}$ be an enumeration of triples $(\LangL, \mathcal{B}, \Delta)$ such that $\LangL$ is a countable language, $\mathcal{B} = \seq{\mathcal{A}_i : i \in \omega}$ is a sequence of $\LangL$-structures with universe $\omega$ and $\Delta$ is a subset of $\operatorname{Fml}(\LangL^+)$ with $\abs{\Delta} < \frakc$.
		Here $\LangL^+ = \LangL \cup \{ c_\alpha : \alpha < \frakc \}$ where the $c_\alpha$'s are new constant symbols and $\operatorname{Fml}(\LangL^+)$ is the set of all $\LangL^+$ formulas with one free variable.
		Here we used the assumption $2^{<\frakc} = \frakc$.
		And ensure each $(\LangL, \mathcal{B}, \Delta)$ occurs cofinally in this sequence.
		
		For $\mathcal{B}_\xi = \seq{\mathcal{A}^\xi_i : i \in \omega}$, put $\mathcal{B}_\xi(i) = (\mathcal{A}^\xi_i, b_0(i), b_1(i), \dots)$, which is a $\LangL^+$-structure.
		
		Let $\seq{X_\xi : \xi < \frakc}$ be an enumeration of $\Pow(\omega)$.
		
		We construct a sequence $\seq{F_\xi : \xi < \frakc}$ of filters inductively so that the following properties hold:
		
		\begin{enumerate}
			\item $F_0$ is the filter consisting of all cofinite subsets of $\omega$.
			\item $F_\xi \subset F_{\xi + 1}$ and $F_\xi  = \bigcup_{\alpha < \xi} F_\alpha$ for $\xi$ limit.
			\item $X_\xi \in F_{\xi+1}$ or $\omega \setminus X_\xi \in F_{\xi+1}$.
			\item $F_\xi$ is generated by $< \frakc$ members.
			\item \label{indhyp} If 
			\begin{equation}
				\text{for all } \Gamma \subset \Delta_\xi \text{ finite, } \{i\in\omega : \Gamma \text{ is satisfiable in } \mathcal{B}_\xi(i) \} \in F_\xi, \label{sh} \tag{$*$}
			\end{equation}
			then there is a $f \in \omega^\omega$ such that for all $\varphi \in \Delta_\xi$, $\{ i \in \omega : f(i) \text{ satisfies } \varphi \text{ in } \mathcal{B}_\xi(i) \} \in F_{\xi+1}$.
		\end{enumerate}
		
		Suppose we have constructed $F_\xi$. We construct $F_{\xi+1}$.
		Let $F_\xi'$ be a generating subset of $F_\xi$ with $\abs{F_\xi'} < \frakc$.
		If (\ref{sh}) is false, let $F_{\xi+1}$ be the filter generated by $F_\xi' \cup \{X_\xi\}$ or $F_\xi' \cup \{\omega \setminus X_\xi\}$.
		Suppose (\ref{sh}).
		
		Put $\P = \operatorname{Fn}(\omega, \omega) = \{ p : \text{$p$ is a finite partial function from $\omega$ to $\omega$}\}$.
		For $n \in \omega$, put
		\[D_n = \{ p \in \P : n \in \dom p \}.\]
		For $A \in F_\xi'$ and $\varphi_1, \dots, \varphi_n \in \Delta_\xi$, put
		\[
		E_{A,\varphi_1,\dots,\varphi_n} = \{ p \in \P : (\exists i \in \dom p \cap A) (p(i) \text{ satisfies } \varphi_1, \dots, \varphi_n \text{ in } \mathcal{B}_\xi(i)) \}.
		\]
		Each $D_n$ is clearly dense.
		In order to show that each $E_{A,\varphi_1,\dots,\varphi_n}$ is dense, take $p \in \P$.
		By (\ref{sh}) and the property $A \in F_\xi$, we can take $i \in A \setminus \dom p$ and $k \in \omega$ such that $k$ satisfies $\varphi_1, \dots, \varphi_n$ in $\mathcal{B}_\xi(i)$.
		Put $q = p \cup \{ (i, k) \}$. This is an extension of $p$ in $E_{A,\varphi_1,\dots,\varphi_n}$.
		
		By using $\operatorname{MA}(\textrm{Cohen})$, take a generic filter $G \subset \P$ with respect to above dense sets. Put $f = \bigcup G$.
		Then $F_\xi'' := F_\xi' \cup \{ Y_\varphi : \varphi \in \Delta_\xi\}$ satisfies finite intersection property, where $Y_\varphi = \{ i \in \omega : f(i) \text{ satisfies } \varphi \text{ in } \mathcal{B}_\xi(i) \}$.
		In order to check this, let $A \in F_\xi'$ and $\varphi_1, \dots, \varphi_n \in \Delta_\xi$.
		Then by genericity, we can take $p \in G \cap E_{A,\varphi_1,\dots,\varphi_n}$.
		So we can take $i \in \dom p \cap A$ such that $p(i) \text{ satisfies } \varphi_1, \dots, \varphi_n \text{ in } \mathcal{B}_\xi(i)$.
		Then we have $i \in A \cap Y_{\varphi_1} \cap \dots \cap Y_{\varphi_n}$.
		
		Let $F_{\xi+1}$ be the filter generated by $F_\xi'' \cup \{X_\xi\}$ or $F_\xi'' \cup \{\omega \setminus X_\xi\}$.
		
		We have constructed $\seq{F_\xi : \xi < \frakc}$.
		In order to check that the resulting ultrafilter $F = \bigcup_{\xi < \frakc} F_\xi$ witnesses $\SAT(\aleph_0)$, let $\LangL$ and $\scrB = \seq{\scrA_i : i \in \omega}$ satisfy the assumption of the theorem.
		Let $\Delta$ be a subset of $\operatorname{Fml}(\LangL^+)$ with $\abs{\Delta} < \frakc$.
		Assume that for all $\Gamma \subset \Delta$ finite, $X_\Gamma := \{i\in\omega : \Gamma \text{ is satisfiable in } \mathcal{B}_\xi(i) \} \in F$.
		By the regularity of $\frakc$, we have $\alpha < \frakc$ such that for all $\Gamma \subset \Delta$ finite, $X_\Gamma \in F_\alpha$.
		Let $\xi \ge \alpha$ be satisfying $(\LangL_\xi, \scrB_\xi, \Delta_\xi) = (\LangL, \scrB, \Delta)$.
		Then by (\ref{indhyp}), there is a $f \in \omega$ such that for all $\varphi \in \Delta$, $\{ i \in \omega : f(i) \text{ satisfies } \varphi \text{ in } \mathcal{B}(i) \} \in F$.
		Thus $\prod_{i \in \omega} \scrA_i / F$ is saturated.

	\end{proof}
	
	\section{$\KT(\aleph_0)$ implies $ \frakc^\exists \le \frakd$}
	
	In this section, we will show the following theorem.
	This proof is based on \cite{shelah1992vive} and \cite{abraham2010proper}.
	
	\begin{thm}\label{ktaleph0impdc}
		$\KT(\aleph_0)$ implies $\frakc^\exists \le \frakd$.
	\end{thm}
	
	\begin{defi}
		Define a language $\LangL$ by $\LangL = \{E, U, V\}$, where $E$ is a binary predicate and $U, V$ are unary predicates.
		We say a $\LangL$-structure $M = (\abs{M}, E^M, U^M, V^M)$ is a {\itshape bipartite graph} if the following conditions hold:
		\begin{enumerate}
			\item $U^M \cup V^M = \abs{M}$,
			\item $U^M \cap V^M = \emptyset$,
			\item $(\forall x, y \in \abs{M}) (x \mathrel{E^M} y \rightarrow (x \in U^M \text{ and } y \in V^M))$.
		\end{enumerate}
	\end{defi}
	
	\begin{defi}
		For $n, k \in \omega$, define a bipartite graph $\Delta_{n,k}$ as follows:
		\begin{enumerate}
			\item $U^{\Delta_{n,k}} = \{ 1, 2, 3, \dots, n \}$
			\item $V^{\Delta_{n,k}} = [\{ 1, 2, 3, \dots, n\}]^{\le k} \setminus \{\emptyset\}$
			\item For $u \in U^{\Delta_{n,k}}, v \in V^{\Delta_{n,k}}$, $u \mathrel{E^{\Delta_{n,k}}} v$ iff $u \in v$.
		\end{enumerate}
	\end{defi}
	
	\begin{defi}
		For $n \in \omega$, Let $G_n = \Delta_{n^3, n}$.
		Let $\Gamma$ be the disjoint union of $(G_n : n \ge 2)$.
		
		We define a natural order $\triangleleft$ on $\Gamma$ by $x \triangleleft y$ if $m < n$ for $x \in G_m, y \in G_n$.
		Then $\Gamma$ is a bipartite graph with an order $\triangleleft$.
		Put $\LangL' = \LangL \cup \{ \triangleleft \}$. From now on, we consider $\LangL'$-structures which are elementarily equivalent to $\Gamma$.
	\end{defi}
	
	\begin{defi}
		Let $\Gamma_\mathrm{NS}$ be a countable non-standard elementary extension of $\Gamma$.
	\end{defi}
	
	When we say connected components, we mean the connected components when we ignore the orientation of the edges.
	
	\begin{lem}
		Let $M$ be an $\LangL'$-structure that is elementarily equivalent to $\Gamma$.
		Then the connected components of $M$ are precisely the maximal antichains of $M$ with respect to $\triangleleft$.
	\end{lem}
	\begin{proof}
		Use elementarity and the fact that any two connected vertexes in $\Gamma$ have path of length at most $4$.
	\end{proof}
	
	Therefore, $\triangleleft$ induces an order on the connected components of $M$ and it is denoted also by $\triangleleft$.
	
	\begin{lem}\label{gammansproperty}
		Every infinite connected component $C$ of $\Gamma_\mathrm{NS}$ satisfies the following:
		\[
		(\forall F \subset C \cap U \text{ finite}) (\exists v \in C \cap V) (v \text{ has an edge to each point in } F).
		\]
	\end{lem}
	\begin{proof}
		Let $F = \{ u_1, \dots, u_n \}$.
		Observe that
		\begin{align*}
			\Gamma \models (\forall x_1) \dots (\forall x_n) & [
			x_1, \dots, x_n \text{ are points in } U \text{ and belong to} \\
			& \text {the same connected component and } \\
			&\text{ the index of this connected component is } \ge n \\
			& \rightarrow (\exists y) [y \text{ belongs to this component, } y \in V \text{ and }  x_1, \dots, x_n \mathrel{E} y]].
		\end{align*}
		By elementarity, $\Gamma_\mathrm{NS}$ satisfies the same formula.
	\end{proof}
	
	\begin{lem}\label{dominatinglem}
		Let $\seq{\Delta_n : n \in \omega}$ be a sequence of bipartite graphs with $\abs{U^{\Delta_n}} = \abs{V^{\Delta_n}} = \aleph_0$.
		Suppose that for each $n \in \omega$,
		\[
		(\forall F \subset U^{\Delta_n} \text{ finite}) (\exists v \in V^{\Delta_n}) (v \text{ has an edge to each point in } F).
		\]
		Then for every ultraproduct $R := \prod_{n \in \omega} \Delta_n / q$, we have
		\[
		(\exists \seq{v_i : i < \frakd} \text{ with each } v_i \in V^R) (\forall u \in U^R) (\exists i < \frakd) (u \mathrel{E^R} v_i).
		\]
	\end{lem}
	\begin{proof}
		We may assume that each $U^{\Delta_n} = \omega$.
		Let $\{ f_i : i < \frakd \}$ be a cofinal subset of $(\omega^\omega, <^*)$.
		For each $n, m \in \omega$, take $v_{n, m} \in V^{\Delta_n}$ that is connected with first $m$ points in $U^{\Delta_n}$.
		For $i < \frakd$, put
		\[
		v_i = [\seq{v_{n, f_i(n)} : n \in \omega}].
		\]
		Let $[u] \in U^R$. Consider $u$ as an element of $\omega^\omega$.
		Take $f_i$ that dominates $u$. Then we have
		\[
		\{ n \in \omega : u(n) \mathrel{E^{\Delta_n}} v_{n, f_i(n)} \} \in q.
		\]
		Therefore $[u] \mathrel{E^R} v_i$.
	\end{proof}
	
	\begin{lem}\label{nonstandardproperty}
		Let $q$ be an ultrafilter over $\omega$ and put $Q = (\Gamma_\mathrm{NS})^\omega/q$.
		Then there exist cofinally many connected components $C$ with respect to $\triangleleft$ such that
		\[
		(\exists \seq{v_i : i < \frakd} \text{ with each } v_i \in C \cap V^Q) (\forall u \in C \cap U^Q) (\exists i < \frakd) (u \mathrel{E^Q} v_i).
		\]
	\end{lem}
	\begin{proof}
		Fix a connected component $C_0$ of $Q$ and $[x_0] \in C_0$.
		Then for each $n \in \omega$, there is an infinite component $C_n$ above $x_0(n)$.
		Now
		\[
		C = \{ [x] \in Q : x \in (\Gamma_\mathrm{NS})^\omega \text{ and } (\forall n \in \omega)(x(n) \in C_n) \}.
		\]
		is a connected component of $Q$ above $C_0$.
		Since $C$ can be viewed as $C = \prod_{n \in \omega} C_n / q$, the conclusion of the lemma holds for $C$ by Lemma \ref{gammansproperty} and Lemma \ref{dominatinglem}.
	\end{proof}
	
	\begin{lem}\label{standardproperty}
		Let $\kappa < \frakc^\exists$ and $p$ an ultrafilter over $\omega$ and put $P = \Gamma^\omega/p$.
		Then for every $C$ in a final segment of connected components of $P$, we have
		\[
		(\forall \seq{v_i : i < \kappa} \text{ with each } v_i \in C \cap V^P) (\exists u \in C \cap U^P) (\forall i < \kappa) (u \not \mathrel{E^P} v_i).
		\]
	\end{lem}
	\begin{proof}
		Let $f \from \omega \to \Gamma$ satisfy $f(n) \in G_n$ for all $n$.
		Let $C_0$ be the connected component that $[f]$ belongs to.
		Take a connected component $C$ such that $C_0 \triangleleft C$ and an element $[g] \in C$.
		Take a function $h \colon \omega \to \omega$ such that $\{ n \in \omega : g(n) \in G_{h(n)} \} \in q$.
		Then $A := \{ n \in \omega : h(n) \ge n \} \in q$.
		Put $h'(n) = \max \{ h(n), n\}$.
		
		Take $\seq{[v_i] : i < \kappa}$ with each $[v_i] \in C \cap V^P$. Then we have
		\[
		B_i := \{ n \in \omega : v_i(n) \in G_{h(n)} \cap V^\Gamma \} \in q.
		\]
		Take $v_i'$ such that $v_i'(n) = v_i(n)$ for $n \in A \cap B_i$ and $v_i'(n) \in [h'(n)^3]^{\le h'(n)}$ for $n\in \omega$.
		The assumption $\kappa < \frakc^\exists$ and the calculation
		\[
		\sum_{n \ge 1} \frac{h'(n)}{h'(n)^3} = \sum_{n \ge 1} \frac{1}{h'(n)^2} \le \sum_{n \ge 1} \frac{1}{n^2} < \infty
		\]
		give a $x \in \prod h'$ such that for all $i < \kappa$, $(\forall^\infty n) (x(n) \not \in v_i'(n))$.
		For each $i < \kappa$, take $n_i$ such that $(\forall n \ge n_i) (x(n) \not \in v_i'(n))$.
		
		Take a point $[u] \in C \cap U^P$ such that $u(n) = x(n)$ for all $n \in A$. Then for all $i < \kappa$ we have
		\[
		\{ n \in \omega : u(n) \not \mathrel{E^\Gamma} v_i(n) \} \supseteq A \cap B_i \cap [n_i, \omega) \in q .
		\]
		Therefore $[u] \not \mathrel{E^P} [v_i]$ for all $i < \kappa$.
	\end{proof}
	
	Assume that $\frakd < \frakc^\exists$.
	Then by Lemma \ref{standardproperty} and Lemma \ref{nonstandardproperty}, for any two ultrafilters $p, q$ over $\omega$, we have $\Gamma^\omega/p \not \simeq (\Gamma_\mathrm{NS})^\omega/q$. So $\neg \KT(\aleph_0)$ holds. We have proved Theorem \ref{ktaleph0impdc}.
	
	\begin{fact}[{{\cite[Lemma 2.3]{klausner2019different}}}]
		$\cov(\nul) \le \frakc^\exists$.
	\end{fact}
	
	\begin{cor}
		In the random model, $\neg \KT(\aleph_0)$ holds. \qed
	\end{cor}
	
	\begin{rmk}
		$\frakv^\forall \le \frakc^\exists$ follows from \cite[Lemma 2.6]{klausner2019different}.
		So the implication $\KT(\aleph_0) \implies \frakd \ge \frakc^\exists$ strengthens the implication $\KT(\aleph_0) \implies \frakd \ge \frakv^\forall$.
	\end{rmk}
	
	\begin{rmk}
		In \cite{shelah1992vive}, Shelah constructed a creature forcing that forces the following statements:
		\begin{enumerate}
			\item There are a finite language $\LangL$ and countable $\LangL$-structures $\mathcal{A}, \mathcal{B}$ with $\mathcal{A} \equiv \mathcal{B}$ such that for all ultrafilters $p, q$ over $\omega$, we have $\mathcal{A}^\omega/p \not \simeq \mathcal{B}^\omega/q$.
			\item There is an ultrafilter $p$ over $\omega$ such that for every countable language $\LangL$ and any sequence $\seq{(\mathcal{A}_n, \mathcal{B}_n) : n \in \omega}$ of pairs of finite $\LangL$-structures, if $\prod_{n \in \omega} \mathcal{A}_n/p \equiv \prod_{n \in \omega} \mathcal{B}_n/p$, then these ultraproducts are isomorphic.
		\end{enumerate}
		Shelah himself pointed out in \cite[Remark 2.2]{shelah1992vive} item 2 holds in the random model. On the other hand, we have proved item 1 also holds in the random model.
		Therefore both of above two statements hold in the random model.
	\end{rmk}
	
	\section{$\KT(\aleph_1)$ in forcing extensions}
	
	A theorem by Golshani and Shelah \cite{golshani2021keislershelah} states that $\cov(\meager) = \frakc \land \cf(\frakc) = \aleph_1$ implies $\KT(\aleph_1)$.
	In \cite{golshani2021keislershelah}, it was also proved that $\cf(\frakc) = \aleph_1$ is not necessary for $\KT(\aleph_1)$.
	In this section, we prove that $\cov(\meager) = \frakc$ is also not necessary for $\KT(\aleph_1)$.
	
	\begin{thm}\label{thm:ktaleph1forcing}
		Let $\lambda > \aleph_1$ be a regular cardinal with $\lambda^{<\lambda} = \lambda$.
		Let $\seq{\P_\alpha, \dot{\Q}_\alpha : \alpha < \omega_1}$ be a finite support forcing iteration.
		Suppose that for all $\alpha < \omega_1$, $\forces_\alpha ``\dot{\Q}_\alpha \text{ is ccc and } \abs{\dot{\Q}_\alpha} \le \lambda"$.
		And suppose that for all even $\alpha < \omega_1$, $\forces_\alpha \dot{\Q}_\alpha = \C_\lambda$.
		Here $\C_\lambda$ denotes the Cohen forcing adjoining $\lambda$ many Cohen reals.
		Then, $\forces_{\omega_1} \KT(\aleph_1)$.
	\end{thm}
	
	\begin{proof}
		Let $G$ be a $(V, \P_{\omega_1})$-generic filter.
		
		Let $\LangL$ be a countable language and $M^0 \equiv M^1$ be two $\LangL$-structures of size $\le \aleph_1$ in $V[G]$.
		Take sequences $\seq{M^l_i : i < \omega_1}$ for $l = 0, 1$ that are increasing and continuous such that each $M^l_i$ is countable elementary substructure of $M^l$ and $M^l = \bigcup_{i < \omega_1} M^l_i$.
		We can take an increasing sequence $\seq{\alpha_i  : i < \omega_1}$ of even ordinals such that $M^l_i \in V[G_{\alpha_i + 1}]$ for every $l < 2$ and $i < \omega_1$.
		
		For $i < \omega_1$ and $\beta < \lambda$, let $c^i_\beta$ be the $\beta$-th Cohen real added by $\dot{\Q}_{\alpha_i}$.
		
		Take an enumeration $\seq{X_\gamma : \gamma < \lambda \cdot \omega_1 }$ of $\Pow(\omega)$ such that $\seq{X_\gamma : \gamma < \lambda \cdot (i+1)} \in V[G_{\alpha_i + 1}]$ for every $i < \omega_1$.
		We can take such a sequence. The reason for this is that we can take $\seq{\dot{X}_\gamma : \lambda \cdot i \le \gamma < \lambda \cdot (i+1)}$ as an enumeration of $\P_{\alpha_i+1}$ nice names for subsets of $\omega$ and put $X_\gamma = (\dot{X}_\gamma)^G$.
		
		For each $l < 2$, take an enumeration $\seq{f^l_\gamma : \gamma < \lambda \cdot \omega_1}$ of $(M^l)^\omega$ such that $f^l_{\lambda \cdot i + \beta} \in (M^l_i)^\omega$ for every $i < \omega_1$ and $\beta < \lambda$ and $\seq{f^l_\gamma : \gamma < \lambda \cdot (i+1)} \in V[G_{\alpha_i + 1}]$.
		
		For $\lambda' < \lambda$, let $G_{\alpha_i, \lambda'}$ denote $G \cap (\P_{\alpha_i} \ast \C_{\lambda'})$.
		
		Now we construct a sequence of quadruples $\seq{(U_\gamma, g^0_\gamma, g^1_\gamma, \lambda_\gamma) : \gamma < \lambda \cdot \omega_1}$ by induction so that the following properties hold.
		
		\begin{enumerate}
			\item Each $U_\gamma$ is a filter over $\omega$.
			\item For every $l < 2$, $i < \omega_1$, $\beta < \lambda$ and $\gamma = \lambda \cdot i + \beta$, $g^l_{\gamma} \in (M^l_i)^\omega \cap V[G_{\alpha_i, \lambda_\gamma}]$.
			\item For every $l < 2$ and $i < \omega_1$, $\seq{g^l_\gamma : \gamma < \lambda \cdot (i+1)} \in V[G_{\alpha_i + 1}]$. \label{seqgappers}
			\item Each $\lambda_\gamma$ is an ordinal below $ \lambda$. 
			For $\lambda \cdot i \le \gamma \le \gamma' < \lambda \cdot (i+1)$, we have $\lambda_\gamma \le \lambda_{\gamma'}$.
			\item For $i < \omega_1$ and $l < 2$, $\{ 	g^l_\gamma : \gamma < \lambda \cdot i \} = \{ 	f^l_\gamma : \gamma < \lambda \cdot i \}$.
			\item If $\lambda \cdot i \le \gamma < \lambda \cdot (i+1)$, then $U_\gamma \in V[G_{\alpha_i}, \lambda_\gamma]$.
			\item If $\gamma < \delta < \lambda \cdot \omega_1$, then $U_\gamma \subset U_\delta$.
			\item If $\gamma < \lambda \cdot \omega_1$ is a limit ordinal, then $U_\gamma = \bigcup_{\delta < \gamma} U_\delta$.
			\item $X_\gamma \in U_{\gamma + 1}$ or $\omega \setminus X_\gamma \in U_{\gamma+1}$.
			\item If $\varphi(x_1, \dots, x_n)$ is a $\LangL$-formula, $\gamma = \lambda \cdot i + \beta$ and $\gamma_1, \dots, \gamma_n \le \gamma$, then $Y_{\varphi, \gamma_1, \dots, \gamma_n}$ defined below belongs to $U_{\gamma+1}$:
			\[
			Y_{\varphi, \gamma_1, \dots, \gamma_n} = \{ k \in \omega : M^0_i \models \varphi(g^0_{\gamma_1}(k), \dots, 	g^0_{\gamma_n}(k)) \Leftrightarrow M^1_i \models \varphi(g^1_{\gamma_1}(k), \dots, g^1_{\gamma_n}(k)) \}
			\] \label{indhyp2}
		\end{enumerate}
		
		\noindent(\textit{Construction}) First we let $U_0$ be the set of cofinite subsets of $\omega$.
		
		Suppose that $\seq{U_\delta : \delta \le \gamma}$ and $\seq{g^0_\delta, g^1_\delta, \lambda_\delta : \delta < \gamma}$ are defined. Now we will define $g^0_\gamma, g^1_\gamma, \lambda_\gamma$ and $U_{\gamma+1}$.
		Take $i$ and $\beta$ such that $\gamma = \lambda \cdot i + \beta$.
		
		Suppose that $\gamma$ is even.
		
		Let $g^0_\gamma = f^0_{\varepsilon_\gamma}$, where $\varepsilon_\gamma$ is the minimum ordinal such that $f^0_{\varepsilon_\gamma}$ does not belong to $\{g^0_\delta : \delta < \gamma \}$.
		
		Take $\lambda' < \lambda$ such that $M^0_i, M^1_i, \seq{g^0_\delta : \delta \le \gamma}, \seq{g^1_\delta : \delta < \gamma} \in V[G_{\alpha_i, \lambda'}]$. Put $\lambda_\gamma = \lambda' + 1$. Take a bijection $\pi^1_i \colon \omega \to M^1_i$ in $V[G_{\alpha_i, \lambda'}]$.
		Define $g^1_\gamma$ by $g^1_\gamma = \pi^1_i \circ c^i_{\lambda'}$.
		
		Put $\mathcal{Y} = \{ Y_{\varphi, \gamma_1, \dots, \gamma_n} :  \varphi(x_1, \dots, x_n) \text{ is a } \LangL\text{-formula and }\gamma_1, \dots, \gamma_n \le \gamma \}$.
		Now we show $U_\gamma \cup \mathcal{Y}$ has the finite intersection property. In order to show it, let $X \in U_\gamma$, $\seq{\varphi_\iota : \iota \in I}$ is a finite sequence of $\LangL$-formulas and $\gamma^\iota_1, \dots, \gamma^\iota_{n_\iota}$ for $\iota \in I$ are ordinals that are less than $\gamma$.
		It suffices to show that the set $D \in V[G_{\alpha_i, \lambda'}]$ defined below is a dense subset of $\C$:
		\begin{align*}
			D = \{ p \in \C:\  & (\exists k \in \dom(p) \cap X)(\forall \iota \in I)  \\
			& \hspace{0.4cm} M^0_i \models \varphi_\iota(g^0_{\gamma^\iota_1}(k), \dots g^0_{\gamma^\iota_{n_\iota}}(k), g^0_\gamma(k)) \Leftrightarrow M^1_i \models \varphi_\iota(g^1_{\gamma^\iota_1}(k), \dots g^1_{\gamma^\iota_{n_\iota}}(k), \pi^1_i(p(k))) \}.
		\end{align*}
		
		We now prove this. Let $p \in \C$.
		
		For each $k \in \omega$ and $\iota \in I$, put
		\[
		v(k, \iota) = \begin{cases}
			1 & \text{if }M^0_i \models \varphi_\iota(g^0_{\gamma^\iota_1}(k), \dots, g^0_{\gamma^\iota_{n_\iota}}(k), g^0_\gamma(k)) \\
			0 & \text{otherwise.}
		\end{cases}
		\]
		And for each $k \in \omega$ put
		\[
		v(k) = \seq{v(k, \iota) : \iota \in I}.
		\]
		Then by finiteness of ${}^I 2$, for some $v_0 \in {}^I 2$, we have $\omega \setminus v^{-1}(v_0) \not \in U_\gamma$.
		
		For each $\iota \in I$, put
		\[
		\varphi_\iota^+(x^\iota_1, \dots, x^\iota_{n_\iota}, y) \equiv
		\begin{cases}
			\varphi_\iota(x^\iota_1, \dots, x^\iota_{n_\iota}, y) & \text{if } v_0(\iota) = 1 \\
			\neg \varphi_\iota(x^\iota_1, \dots, x^\iota_{n_\iota}, y) & \text{otherwise.}
		\end{cases}
		\]
		Put
		\[
		\psi \equiv \exists y \bigwedge_{\iota \in I} \varphi_\iota^+(x^\iota_1, \dots, x^\iota_{n_\iota}, y).
		\]
		Then by the induction hypothesis (\ref{indhyp2}), $Y_{\psi, \seq{\gamma^\iota_1, \dots \gamma^\iota_{n_\iota} : \iota \in I}} \in U_\gamma$.
		So take $k \in X \cap v^{-1}(v_0) \cap Y_{\psi, \seq{\gamma^\iota_1, \dots \gamma^\iota_{n_\iota} : \iota \in I}} \setminus \dom(p)$.
		
		Since $M^0_i \models \psi(\seq{g^0_{\gamma^\iota_1}(k), \dots g^0_{\gamma^\iota_{n_\iota}}(k) : \iota \in I})$, we have $M^1_i \models \psi(\seq{g^1_{\gamma^\iota_1}(k), \dots g^1_{\gamma^\iota_{n_\iota}}(k) : \iota \in I})$.
		
		By the definition of $\psi$, we can take $y \in M^1_i$ such that $M^1_i \models \varphi_\iota^+(g^1_{\gamma^\iota_1}(k), \dots, g^1_{\gamma^\iota_{n_\iota}}(k), y)$ for every $\iota \in I$.
		We now put $q = p \cup \{ (k, (\pi^1_i)^{-1}(y)) \} \in \C$. This witnesses denseness of $D$.
		
		Now we define $U_{\gamma+1}$ as the filter generated by $U_\gamma \cup \mathcal{Y} \cup \{ X_\gamma \}$ or the filter generated by $U_\gamma \cup \mathcal{Y} \cup \{ \omega \setminus X_\gamma \}$.

		When $\gamma$ is odd, do the same construction above except for swapping 0 and 1.
		Since the above construction below $\lambda \cdot (i+1)$ can be performed in $V[G_{\alpha_i + 1}]$, (\ref{seqgappers}) in the induction hypothesis holds.
		(\textit{End of Construction.})
		
		Now we put $U = \bigcup_{\gamma < \lambda \cdot \omega_1} U_\gamma$, which is an ultrafilter over $\omega$.
		Then the function
		\[
		\seq{([g^0_\gamma]_U, [g^1_\gamma]_U) : \gamma < \lambda \cdot \omega_1}
		\]
		witnesses $(M^0)^\omega / U \simeq (M^1)^\omega / U$.
	\end{proof}

	\begin{cor}
		$\Con(\ZFC) \rightarrow \Con(\ZFC + \cof(\nul) = \aleph_1 < \frakc + \KT(\aleph_1))$.
	\end{cor}
	\begin{proof}
		Let $\mathbb{A}$ denote the amoeba forcing.
		Letting
		\[
		\forces_\alpha \dot{\Q}_\alpha = \mathbb{A}
		\]
		for every odd $\alpha$ gives the consistency.
	\end{proof}
	
	\section{Open questions}
	
	\begin{question}
		\begin{enumerate}
			\item Does $\KT(\aleph_1)$ imply a stronger hypothesis than $\mcf = \aleph_1$? In particular does $\KT(\aleph_1)$ imply $\non(\meager) = \aleph_1$?
			\item Does $\KT(\aleph_0)$ imply a stronger hypothesis than $\frakc^\exists \le \frakd$? In particular does $\KT(\aleph_0)$ imply $\non(\meager) \le \cov(\meager)$?
		\end{enumerate}
	\end{question}
	
	\section*{Acknowledgement}
	
	The author is grateful to his current supervisor Yasuo Yoshinobu and his future supervisor Jörg Brendle. Both of them gave him helpful comments. In particular, the idea of the proof of Theorem \ref{thm:ktaleph1forcing} is due to Brendle.
	
	\nocite{*}
	\printbibliography
	
\end{document}